\documentclass[12pt,reqno]{amsart}

\usepackage{amsfonts,amsmath,amsthm}
\usepackage{amssymb,epsfig}
\usepackage{enumerate} 
\usepackage[text={425pt,650pt},centering]{geometry}
\usepackage{graphicx}
\usepackage{epsfig}
\usepackage{tikz}
\usepackage{tikz-cd}
\usepackage{caption}
\usepackage{array}
\usepackage{varwidth}
\usepackage{color} %color
\definecolor{vert}{rgb}{0,0.6,0}

\numberwithin{figure}{section}

\theoremstyle{plain}
\newtheorem{thm}{Theorem}[section]

\newtheorem{defn}{Definition}

\newtheorem{ex}{Example}
\newtheorem{lem}[thm]{Lemma}
\newtheorem{cor}[thm]{Corollary}
\newtheorem{prop}[thm]{Proposition}
\theoremstyle{remark}
\newtheorem{rem}{\bf{Remark}}
\numberwithin{equation}{section}

\newcommand{\N}{\mathbb{N}}

\newcommand{\R}{\mathbb{R}}

\newcommand{\gam}{\gamma}

\newcommand{\e}{\varepsilon}

\newcommand{\sig}{\sigma}

\newcommand{\Lam}{\Lambda}

\begin{document}

\title[Optimal convergence rate in the convex infinite-dimensional setting]
{Optimal Convergence Rate for Periodic Homogenization of Rearrangement-Invariant Convex Hamilton--Jacobi Equations in Infinite Dimensions}

\author[Seho Park]
{Seho Park}

% \thanks{
% ABC
% }

\address[Seho Park]
{
Department of Mathematics, 
University of Wisconsin Madison, 480 Lincoln  Drive, Madison, WI 53706, USA}
\email{park646@wisc.edu}

\date{\today}

\maketitle
\begin{abstract}
We prove the optimal convergence rate $O(\e)$ for periodic homogenization of convex Hamilton--Jacobi equations arising from infinite systems of indistinguishable particles on the torus, under the assumption that the initial data depend only on the mean configuration. This extends the finite-dimensional result \cite{Tran-Yu-optimal}, which is based on the large-time behavior of the Lagrangian action metric and a curve-surgery argument. Here, these tools cannot be applied directly because minimizing curves live in an infinite-dimensional Hilbert space, where local compactness and finite-dimensional topology are unavailable. We overcome this difficulty by cutting the finite-dimensional mean-time projection of a minimizing curve and gluing the lifted pieces in the Hilbert space using the compact quotient induced by periodicity and rearrangement invariance. We conclude with an example showing that this rate is sharp.
\end{abstract}

\section{Introduction}
We begin with a brief review of the homogenization of Hamilton--Jacobi equations, with an emphasis on quantitative convergence results. In finite dimensions, the periodic homogenization of first-order Hamilton--Jacobi equations with oscillatory spatial dependence was initiated by Lions, Papanicolaou, and Varadhan \cite{LPV} and further developed by Evans \cite{Evans-perturb}. For general nonconvex Hamiltonians, the best classical convergence rate is $O(\e^{1/3})$, obtained by Capuzzo-Dolcetta and Ishii \cite{CDI-nonconvex-rate}. In the convex setting, sharper estimates were obtained using optimal control formulas, weak KAM theory and Aubry--Mather theory~\cite{MTY-rate-2dim}, and later through the optimal control framework in~\cite{MTY-rate-2dim} combined with Alexander's theorem~\cite{cooperman}. Nevertheless, these approaches did not establish the optimal $O(\varepsilon)$ convergence rate for general convex Hamiltonians in all dimensions. The general convex case was finally resolved by Tran and Yu \cite{Tran-Yu-optimal}, who extended Burago's convergence estimate for stable norms in metric geometry \cite{Burago} to the homogenization of Hamilton--Jacobi equations. Their proof uses the large-time asymptotics of the Lagrangian action metric, obtained through subadditivity and superadditivity estimates derived from a curve-cutting argument. This metric viewpoint has since been adapted to multiscale and space-time periodic Hamiltonians~\cite{Han-Jang-multiscale-2023, Nguyen-rate-2024}, to state-constraint, Neumann, and Dirichlet problems on perforated domains~\cite{Han-Jing-Mitake-Tran-state-constraint, Mitake-Ni-Neumann, Han-Tu-Dirichlet}, and to Hamiltonians that are periodic in the unknown~\cite{Mitake-Panrui-Tran-2025}. Related quantitative extensions to nonlinear weakly coupled systems were obtained by Mitake and Ni~\cite{Mitake-Ni-weakly-coupled}.
 
In infinite dimensions, the well-posedness theory for viscosity solutions of the first-order Hamilton--Jacobi equations was developed in the foundational works of Crandall and Lions \cite{C-L-1, C-L-2} and Ishii \cite{Ishii}. Infinite-dimensional Hamilton--Jacobi equations have also been studied from the viewpoint of optimal control, where dynamic programming principles and value-function representations provide Lagrangian action formulas in Hilbert spaces~\cite{barbu,barbu1983existence,barbu1985,cannarsa,Gomes-Nurbekyan-CalculusofVariation}. In the setting of infinite systems of indistinguishable particles, Gomes and Nurbekyan studied the cell problem for convex Hamiltonians \cite{Gomes-Nurbekyan-weakKAM}, building on the infinite-dimensional weak KAM framework developed by Gangbo and Tudorascu \cite{Gangbo-Tudorasco-1, Gangbo-Tudorasco-2}. 
Related qualitative homogenization results for action functionals on Wasserstein space were obtained by Gangbo and Tudorascu~\cite{Gangbo-Tudorasco-2012-homogenization}.
However, quantitative convergence results for infinite-dimensional homogenization remain limited. More recently, Park~\cite{park2026} proved qualitative homogenization and obtained an $O(\e^{1/3})$ convergence rate for possibly nonconvex Hamiltonians in this rearrangement-invariant infinite-dimensional setting with initial data depending only on the mean configuration.

Very recently, Ding, Ekren, Han, and Zitridis~\cite{Ding-Ekren-Han-Zitridis} established quantitative homogenization directly on the Wasserstein space \(\mathcal P_2(\mathbb R^d)\). A major difficulty in their setting is that $\mathcal P_2(\mathbb R^d)$ is not locally compact, in contrast to the compact space $\mathcal P_2(\mathbb T^d)$. They obtain an $O(\sqrt{\varepsilon})$ rate for general multiscale Hamiltonians and the sharp $O(\varepsilon)$ rate when the Hamiltonian depends only on the fast variable and the momentum. Their terminal cost may depend on the entire terminal distribution rather than only on its mean. In the sharp-rate regime, for fixed initial and terminal distributions, the running cost can be expressed as the infimum, over all couplings of these distributions, of the expected finite-dimensional minimal action. This allows the finite-dimensional metric estimate of~\cite{Tran-Yu-optimal} to be applied pointwise, integrated with respect to each coupling, and then minimized over all couplings.

The present paper addresses a complementary infinite-dimensional difficulty to obtain the optimal $O(\e)$ rate under convexity. Our Hamiltonian may be a general rearrangement-invariant functional of the fast variable and momentum and need not arise by integrating a particlewise Hamiltonian. Consequently, the corresponding Lagrangian action cannot in general be reduced to finite-dimensional costs. Instead, the mean-dependent initial-data assumption identifies a finite-dimensional macroscopic variable while the minimizing curves themselves remain in the full Hilbert space. We introduce a mean-endpoint action metric and develop a hybrid curve-surgery argument: the curve is cut through its finite-dimensional mean-time projection, while the resulting pieces are transformed and glued in the full Hilbert space using the compact quotient induced by periodicity and rearrangement invariance. This underlying principle may provide a useful template for other infinite-dimensional variational problems. Thus, \cite{Ding-Ekren-Han-Zitridis} and the present paper advance quantitative homogenization in complementary infinite-dimensional settings and overcome the lack of compactness by different mechanisms.

\subsection{Setting and Motivation}
Following~\cite{Gomes-Nurbekyan-weakKAM}, we recall the mechanical system for a continuum of indistinguishable particles on the \(d\)-dimensional torus. Let $I=[0,1]^d$ be equipped with the Lebesgue measure \(\lambda_0\). The particle labels are parametrized by points \(i\in I\), and a configuration is represented by a map
\[
    x\in V:=L^2(I;\mathbb R^d).
\]
For \(i\in I\), the value \(x(i)\in\mathbb R^d\) represents the position of the particle labeled by \(i\). 
We consider a Hamiltonian
\[
    H=H(x,p):V\times V\to\mathbb R.
\]
For the momentum variable, we identify \(V^*\) with \(V\) since $V$ is a Hilbert space.

The physical structure of the problem imposes two symmetries on \(H\). First, since the particles live on the torus, the Hamiltonian is periodic in the configuration variable. More precisely, if
\[
    \Lambda:=L^2(I;\mathbb Z^d),
\]
then $x \mapsto H(x,p)$ is $\Lam$-periodic.

Also, since the particles are indistinguishable, the Hamiltonian is invariant under rearrangements of the particles. These rearrangements are described by the group of measure-preserving bijections of \(I\), namely
\[
    \mathcal G
    :=
    \left\{
    g:I\to I:
    g \text{ is bijective},\
    g,g^{-1}\text{ are Borel measurable},\
    g_{\#}\lambda_0=(g^{-1})_{\#}\lambda_0=\lambda_0
    \right\}.
\]
For \(g\in\mathcal G\), the configuration \(x\circ g\) represents the same particle configuration as \(x\), up to relabeling.

We next consider the periodic homogenization problem for the Hamilton--Jacobi equation on \(V\). For each \(\varepsilon>0\), let \(u^\varepsilon\) be the viscosity solution of
\begin{equation*}
\mathrm{(CP)_\e} \quad\begin{cases}
    u^\varepsilon_t
    +
    H\left(\dfrac{x}{\varepsilon},Du^\varepsilon\right)=0
        &\text{in } V\times(0,\infty),\\[0.8em]
    u^\varepsilon(x,0)=u_0(x)
        &\text{on } V.
\end{cases}
\end{equation*}
Here, $H$ satisfies the following assumptions. 
\begin{align*}
    &(\mathrm{H1})\quad H(x + z, p) = H(x,p) &&\text{(Periodicity)} \\
    &(\mathrm{H2})\quad H(x \circ g, p\circ g) = H(x,p) &&\text{(Rearrangement Invariance)} \\
    &(\mathrm{H3})
    \begin{cases}
        |H(x,p) - H(y,p)| \le L_x(1+ \|p\|)\|x-y\|\\
        |H(x,p) - H(x,q)| \le L_R \|p-q\|, \quad p,q \in B(0,R)
    \end{cases} &&\text{(Locally Lipschitz)}\\
    &(\mathrm{H4})\quad \lim_{\|p\| \to \infty} \inf_{x \in V} H(x,p) = +\infty &&\text{(Coercivity)}
\end{align*}
Conditions \rm{(H1)}--\rm{(H2)} hold for all $x,p\in V$, $z\in\Lambda$, and $g\in G$. In \rm{(H3)}, there exists $L_x>0$ such that the first estimate holds for all $x,y,p\in V$, and, for every $R>0$, there exists $L_R>0$ such that the second estimate holds for all $x\in V$ and $p,q\in B_V(0,R)$.

Let $\chi_I\in V$ denote the constant function equal to $1$ on $I$. For $x\in V$, define its mean and its projection onto the constant configurations (its mean configuration) by
\[
    \mathfrak m(x):=\int_I x\,d\lambda_0\in\mathbb R^d,
    \qquad
    Mx:=\mathfrak m(x)\chi_I\in V.
\]
For the initial data, we assume:
\begin{align*}
(\mathrm{I1})\quad & u_0(x) = u_0(Mx) \\
(\mathrm{I2})\quad & u_0 \in C_b^1(V)
\end{align*}
Note that the initial data depend only on the mean configuration.

Under these assumptions, $u^\e\to u$ locally uniformly as $\e\to0$. Moreover, the homogenized limit depends only on the mean configuration: 
\[
    u(x,t)=\widetilde u(\mathfrak m(x),t).
\]
Moreover, the finite-dimensional limit $\widetilde u: \R^d \times [0,\infty) \to\R$ solves an effective Hamilton--Jacobi equation: 
\[
{\rm(\overline{CP})} \quad
\begin{cases}
\widetilde u_t + \overline{H}(D \widetilde u) = 0
& \text{in } \R^d \times (0,\infty), \\[0.8ex]
\widetilde u(x,0)= \widetilde u_0(x)
& \text{on } \R^d
\end{cases}
\]
where $u_0 = \widetilde u_0 \circ \mathfrak m$.
See \cite{park2026} for details. Thus, although the oscillatory problem \((\mathrm{CP}_\varepsilon)\) is posed on the infinite-dimensional space \(V\), the effective dynamics are finite-dimensional. Here, the effective Hamiltonian in the mean direction is given by the cell problem
\[
    H(y,p\chi_I+Dv(y))=\overline H(p)
    \qquad\text{in } V.
\]

Quantitatively, in the general nonconvex setting, the known convergence rate is $O(\e^{1/3})$:
\[
    \sup_{x\in V,\ 0\le t\le T}
    |u^\varepsilon(x,t)-\widetilde u(\mathfrak m(x),t)|
    \le
    C_T\varepsilon^{1/3}.
\]
This rate coincides with the general finite-dimensional periodic homogenization rate for nonconvex Hamiltonians \cite{CDI-nonconvex-rate}.

\subsection{Main result}
The goal of this paper is to establish the optimal $O(\varepsilon)$ convergence rate in the present infinite-dimensional setting, under the additional assumption that the Hamiltonian is convex in the momentum variable. To this end, we further assume:

\begin{align*}
(\mathrm{H5}) \quad \text{for } x\in V, \; p \mapsto H(x,p) \quad \text{is convex}.
\end{align*}

Here is our main result.
\begin{thm}
    Assume \emph{(H1)-(H5)} and \emph{(I1)-(I2)}. For $\e > 0$, let $u^\e$ be the viscosity solution to $\mathrm{(CP)_\e}$. Let $\widetilde u$ be the viscosity solution to $\mathrm{(\overline{CP})}$. Then there exists $C > 0$ depending only on $H$ and $\|Du_0\|_{L^\infty(V)}$ such that
    \begin{equation*}
        \|u^\e(x,t) - \widetilde u(\mathfrak m(x),t)\|_{L^\infty(V \times [0,\infty))} \le C\e.
    \end{equation*}
\end{thm}

The remainder of the paper is organized as follows. In Section~\ref{section2}, we reduce the quantitative homogenization problem to the large-scale behavior of a mean-endpoint action metric and outline the curve-surgery argument. In Section~\ref{section3}, we establish the almost-subadditivity and almost-superadditivity estimates. Section~\ref{section4} constructs the homogenized metric, establishes a quantitative error bound, and identifies the metric with the effective Lagrangian. Finally, in Section~\ref{section5}, we combine these results with the variational representations of the oscillatory and effective solutions to prove Theorem~1.1 and discuss the optimality of the rate.

\section{Sketch of the proof}\label{section2}

In this section, we first recall the qualitative homogenization results and explain how the optimal convergence problem reduces to a metric problem with mean-endpoint constraints. We then introduce the homogenized metric arising from almost-subadditivity and almost-superadditivity estimates, and show how this yields the optimal convergence rate. The key estimates are obtained through infinite-dimensional curve-surgery arguments.

\subsection{Reduction to the mean-endpoint metric}
We first recall the uniform estimates obtained from the qualitative homogenization theory in \cite{park2026}. These estimates are the starting point for the optimal-rate argument: they allow us to replace the full configuration \(x\in V\) by its mean configuration \(Mx\), with only an \(O(\varepsilon)\) error. See Propositions~2.5 and 2.8 in~\cite{park2026} for details.

\begin{lem}[Uniform Lipschitz estimate and mean reduction] There exists \(C>0\), independent of \(\varepsilon\), such that 
\begin{equation*} 
|u^\varepsilon(x,t)-u^\varepsilon(y,s)| \le C(\|x-y\|_{L^2}+|t-s|) 
\end{equation*} 
for all \(x,y\in V\) and \(s,t\ge0\). Moreover, 
\begin{equation*} 
|u^\varepsilon(x,t)-u^\varepsilon(Mx,t)| \le C\varepsilon 
\end{equation*} 
for all \((x,t)\in V\times[0,\infty)\). The homogenized limit \(u\) satisfies
\begin{equation*} 
u(x,t)=u(Mx,t). 
\end{equation*} 
\end{lem}\label{lemma:equi-Lip}

Lemma~\ref{lemma:equi-Lip} implies that only the values of \(H\) on a bounded set of momenta are relevant. Therefore, after modifying \(H\) outside a sufficiently large ball in the momentum variable if necessary, we may assume without loss of generality that \(H\) has quadratic growth in \(p\) as stated in the following lemma.

\begin{lem}
\label{lem:quadratic-reduction}
Let $C$ be the spatial Lipschitz constant in Lemma~\ref{lemma:equi-Lip}. There exists a Hamiltonian $\widetilde H:V\times V\to\mathbb R$ satisfying {\rm (H1)--(H5)} such that
\[
\widetilde H(x,p)=H(x,p)
\qquad\text{whenever }\, \|p\|_{L^2}\leq C,
\]
and, for some $K_0>1$,
\begin{equation}\label{eq: quadratic growth}
\frac12\|p\|_{L^2}^2-K_0
\leq \widetilde H(x,p)
\leq \frac12\|p\|_{L^2}^2+K_0
\qquad\text{for all }(x,p)\in V\times V.
\end{equation}
Moreover, replacing $H$ by $\widetilde H$ does not change the viscosity solutions $u^\varepsilon$ and their homogenized limit.
\end{lem}

\begin{proof}
Choose $R>C$ and let $K_R$ be a Lipschitz constant of $H$ on
$V\times B(0,R)$. By {\rm (H1)} and {\rm (H3)},
\[
C_R:=\sup_{x\in V,  \|p\|_{L^2}\leq R}|H(x,p)|<\infty.
\]
See Lemma~2.2 in \cite{park2026} for details.

Define the convex Lipschitz extension
\begin{equation}\label{eq: s2_Hr}
H_R(x,p):=
\inf_{\|q\|_{L^2}\leq R}
\left\{H(x,q)+K_R\|p-q\|_{L^2}\right\}.
\end{equation}
Note that $H_R (x,p) = H(x,p)$ on $V\times B(0,R)$.
Define
\begin{equation*}
    \eta_R(q) := \begin{cases}
        0, \qquad &\|q\|_{L^2} \le R,\\
        +\infty,\qquad &\|q\|_{L^2} > R.
    \end{cases}
\end{equation*}
For each $x \in V$, since \eqref{eq: s2_Hr} is the infimal convolution of the convex functions
\[
q\mapsto H(x,q)+\eta_R(q), \qquad \text{and} \qquad q\mapsto K_R\|q\|_{L^2},
\]
$p\mapsto H_R(x,p)$ is convex. It is also $K_R$-Lipschitz in $p$. By construction, we also have
\[
H_R(x+z,p)=H_R(x,p),
\qquad
H_R(x\circ g,p\circ g)=H_R(x,p)
\]
for every $z\in\Lambda$ and $g\in\mathcal G$. The local Lipschitz regularity in $x$ follows directly from {\rm (H3)}.

Choose
\[
    A\geq C_R+\frac12R^2
\]
and set
\[
    V(p):=\frac12\|p\|_{L^2}^2-A,
    \qquad
    \widetilde H(x,p):=\max\{H_R(x,p),V(p)\}
\]
for $x,p \in V$. If $\|p\|_{L^2}\leq R$, then
\[
    V(p)\leq -C_R\leq H(x,p)=H_R(x,p),
\]
so $\widetilde H=H$ on $V\times B(0,R)$. Since the maximum of two convex functions is convex, $\widetilde H$ remains convex in $p$. It also preserves {\rm (H1)--(H3)}, and it is coercive because $\widetilde H\geq V$.

Since $H_R(x,0)=H(x,0)$ and $H_R$ is $K_R$-Lipschitz in $p$,
\[
    H_R(x,p)\leq C_R+K_R\|p\|_{L^2}
    \leq \frac12\|p\|_{L^2}^2+C_R+\frac12K_R^2.
\]
Together with $\widetilde H\geq V$, this gives $K_0>0$ such that
\[
    \frac12\|p\|_{L^2}^2-K_0
    \leq\widetilde H(x,p)
    \leq\frac12\|p\|_{L^2}^2+K_0.
\]

Finally, every spatial gradient occurring in the viscosity inequalities for the $C$-Lipschitz function $u^\varepsilon(\cdot,t)$ has norm at most $C$. Since $R>C$ and $\widetilde H=H$ on $V\times B(0,R)$, uniqueness therefore shows that the oscillatory solutions are unchanged. Consequently, their homogenized limit is unchanged as well.
\end{proof}

By Lemma~\ref{lem:quadratic-reduction}, we assume that $H$ satisfies~\eqref{eq: quadratic growth}.

For every $(x,v) \in V \times V$, let
\[
    L(x,v) :=
    \sup_{p\in V}
    \left\{
    \langle p,v\rangle_{L^2}-H(x,p)
    \right\}
\]
be the Lagrangian associated with \(H\). It follows that:
\begin{equation*}
    \frac12\|v\|_{L^2}^2-K_0 \le L(x,v) \le \frac12\|v\|_{L^2}^2+K_0 \qquad \text{for all } (x,v) \in V\times V.
\end{equation*}
Furthermore, the periodicity and rearrangement invariance of
\(H\) pass to \(L\):
\begin{align*}
    L(x+z,v)=L(x,v)&,\qquad z\in\Lambda\\
    L(x\circ g,v\circ g)=L(x,v)&,\qquad g\in\mathcal G.
\end{align*}
These invariances will be used later to rearrange and translate curve pieces without changing their action.

Since the modified Hamiltonian is convex and has two-sided quadratic growth
in the momentum variable, the corresponding value function is the unique
viscosity solution of the Cauchy problem; see, for example,
\cite{C-L-2, ishii_optimalcontrol, Gomes-Nurbekyan-CalculusofVariation}. Consequently, for $(x,t) \in V\times (0,\infty)$ the solution $u^\e (x,t)$ to $\mathrm{(CP)_\e}$ admits the optimal-control representation
\begin{equation}\label{eq:optimal control_ue}
    u^\varepsilon(x,t)
    =
    \inf_{\substack{\eta\in AC([0,t];V)\\ \eta(t)=x}}
    \left\{
    u_0(\eta(0))
    +
    \int_0^t
    L\left(\frac{\eta(s)}{\varepsilon},\dot\eta(s)\right)\,ds
    \right\}.
\end{equation}
After the change of variables $\eta(s)=\varepsilon\gamma(s/\varepsilon)$, \eqref{eq:optimal control_ue} becomes
\begin{equation}\label{eq:optimalcontrol_ue_2}
    u^\varepsilon(x,t)
    =
    \inf_{\substack{\gamma\in AC([0,t/\e];V)
    \\\gamma(t/\varepsilon)=x/\varepsilon}}
    \left\{
    u_0(\varepsilon\gamma(0))
    +
    \varepsilon
    \int_0^{t/\varepsilon}
    L(\gamma(\tau),\dot\gamma(\tau))\,d\tau
    \right\}.
\end{equation}

Lemma~\ref{lemma:equi-Lip} reduces the convergence-rate analysis for \(u^\varepsilon\) on \(V\) to terminal points lying in the finite-dimensional subspace of constant configurations. Indeed, since both \(u^\varepsilon(x,t)\) and the homogenized limit \(u(x,t)\) can be compared with their values at the mean configuration \(Mx\), it suffices to estimate the oscillatory value function in \eqref{eq:optimalcontrol_ue_2} with a prescribed mean endpoint.
This motivates the  mean-restricted value function
\begin{equation}\label{def_Ue}
    U^\varepsilon(q,t)
    :=
    \inf_{\mathfrak m(x) = q} u^\varepsilon(x,t),
    \qquad q\in\mathbb R^d.
\end{equation}
By the mean-reduction estimate in Lemma~\ref{lemma:equi-Lip},
\begin{equation} \label{eq: mean reduction estimate}
    |u^\varepsilon(x,t)-U^\varepsilon(\mathfrak m(x),t)|
    \le C\varepsilon, \qquad x \in V.
\end{equation}

The role of the mean-restricted value function \(U^\varepsilon\) is to isolate the finite-dimensional macroscopic variable while preserving the microscopic oscillations in the Lagrangian action. Thus, the convergence-rate analysis for \(u^\varepsilon\) reduces to comparing \(U^\varepsilon\) with the homogenized limit \(\widetilde u\). Since the assumption \rm{(I1)} implies that the initial data depend only on the mean configuration, \(U^\varepsilon\) depends only on the initial and terminal means. This motivates the introduction of the following mean-endpoint metric.

\begin{defn}[Mean-Endpoint Metric]
    For \(t>0\) and \(a,q\in\mathbb R^d\),
    \begin{equation*}
        m(t,a,q) := \inf_{\gamma\in AC([0,t];V)} \left\{ \int_0^t L(\gamma(s),\dot\gamma(s))\,ds: \mathfrak m(\gamma(0))=a, \ \mathfrak m(\gamma(t))=q \right\}.
    \end{equation*} 
\end{defn}

Here, \(m(t,a,q)\) denotes the infimum of the action over all trajectories of duration \(t\) whose initial and terminal configurations have means \(a\) and \(q\), respectively.
From \eqref{eq:optimalcontrol_ue_2} and \eqref{def_Ue}, \(U^\varepsilon\) can be
written as
\begin{equation}\label{eq:Ue representation}
    U^\varepsilon(q,t)
    =
    \inf_{a\in\mathbb R^d}
    \left\{
    \widetilde u_0(a)
    +
    \varepsilon
    m\left(
    \frac{t}{\varepsilon},
    \frac{a}{\varepsilon},
    \frac{q}{\varepsilon}
    \right)
    \right\}, \qquad q\in\mathbb R^d.
\end{equation}
Thus, the quantitative convergence problem for \(u^\varepsilon\) reduces to the large-scale analysis of the mean-endpoint metric \(m\).

\subsection{Homogenized metric and optimal convergence}
The next step is to show that the mean-endpoint metric admits a deterministic large-scale limit. As in the finite-dimensional convex theory, this follows from almost-subadditivity and almost-superadditivity estimates. In the present setting, these estimates take the dyadic form
\begin{equation}\tag{almost-subadditivity}
    m(2t,0,2q)\le 2m(t,0,q)+C_R
\end{equation}
and
\begin{equation}\tag{almost-superadditivity}
    2m(t,0,q)\le m(2t,0,2q)+C_R
\end{equation}
whenever \(q\in \mathbb R^d\) and \(|q|\le Rt\). The proofs of these estimates are carried out in Section~\ref{section3}. Together, they imply that the rescaled actions stabilize up to a bounded error. This allows us to define the homogenized mean-endpoint metric:
\begin{equation*}
    \overline m(t,a,q)
    :=
    \lim_{k\to\infty}\frac1k m(kt,ka,kq).
\end{equation*}
More precisely, they yield the estimate
\begin{equation}\label{eq: estimate m and mbar}
    \left|
    \varepsilon
    m\left(
    \frac{t}{\varepsilon},
    \frac{a}{\varepsilon},
    \frac{q}{\varepsilon}
    \right)
    -
    \overline m(t,a,q)
    \right|
    \le C_R\varepsilon
\end{equation}
whenever \(a,q\in\mathbb R^d\) and \(|q-a|\le Rt\). We then identify \(\overline m\) with the effective Lagrangian: for every $a,q \in \R^d$,
\[
    \overline m(t,a,q) =
    t\overline L\left(\frac{q-a}{t}\right),
    \qquad
    \overline L=\overline H^*.
\]
We carry out this identification in Section~\ref{section4}.

Since the homogenized limit \(\widetilde u\) solves $(\overline{\mathrm{CP}})$, the Lax--Oleinik formula gives the representation
\begin{equation}\label{eq:HopfLax_u}
\widetilde u(q,t) = \inf_{a\in\mathbb R^d} \left\{\widetilde u_0(a) + t\overline L\left(\frac{q-a}{t}\right) \right\} = \inf_{a\in\mathbb R^d} \left\{\widetilde u_0(a)+\overline{m}(t,a,q) \right\}
\end{equation}
for every $q \in \R^d$. Combining \eqref{eq:Ue representation}, \eqref{eq: estimate m and mbar}, and \eqref{eq:HopfLax_u} gives
\[
    |U^\varepsilon(q,t)-\widetilde u(q,t)|
    \le C\varepsilon
\]
for every $q \in \R^d$.
Finally, \eqref{eq: mean reduction estimate} transfers the estimate back to the full solution \(u^\varepsilon\), giving the optimal convergence rate $O(\e)$:
\[
    \|u^\varepsilon(x,t) -\widetilde u(\mathfrak m(x),t)\|_{L^\infty(V \times [0,\infty))}
    \le C\varepsilon.
\]
See Section~\ref{section5} for details and an example showing the optimality of the rate.

\subsection{Infinite-dimensional curve-surgery}
We now explain the geometric mechanism behind the almost-subadditivity and almost-superadditivity estimates. The key idea is a curve-surgery argument: starting from one or more nearly minimizing curves, we cut them into suitable pieces, transform the pieces using the symmetries of the problem, and glue them back together by inserting short connectors.

For almost-subadditivity, the surgery is relatively simple. Starting from two near minimizers for $m(t,0,q)$ for fixed $q \in \R^d$, we translate and rearrange the second curve so that its initial point is close to the endpoint of the first curve, and then join the two curves by a short connector. The resulting curve is admissible for $m(2t,0,2q)$, thereby yielding the almost-subadditivity estimate. Since the original curves already use the full allotted time, we first compress a small subinterval of one curve to save a fixed amount of time for the connector. The quadratic bounds on \(L\) ensure that the resulting increase in action is uniformly bounded. See Section~\ref{section3} for details.

The almost-superadditivity estimate is more delicate. Starting from a near minimizer $\gamma:[0,2t] \to V$ for \(m(2t,0,2q)\), we need to split the curve into two families of pieces, each carrying mean displacement \(q\) over total time \(t\). This is precisely where the finite-dimensional curve-cutting lemma enters. We use the following topological lemma of Burago \cite{Burago}.

\begin{lem}\label{lem: burago}
Let \(\xi:[0,T]\to\mathbb R^m\) be continuous. Then there exist finitely many pairwise disjoint intervals 
\[ 
[a_i,b_i]\subset[0,T], \qquad i=1,\dots,k, 
\] 
with 
\[ 
k\le \frac{m+1}{2}, 
\] 
such that 
\[ 
\sum_{i=1}^k \bigl(\xi(b_i)-\xi(a_i)\bigr) = \frac{\xi(T)-\xi(0)}{2}. 
\] 
\end{lem}
In finite dimensions, this lemma is applied directly to the minimizing curve together with the time coordinate. In the present setting, the curve itself takes values in the infinite-dimensional Hilbert space \(V\), so the lemma cannot be applied directly to \(\gamma\). The key point is that the metric \(m\) imposes only mean endpoint constraints. Therefore the relevant finite-dimensional object is the mean-time path
\[
    \xi(s):=(\mathfrak m(\gamma(s)),s), \qquad 0\le s \le 2t.
\]
Applying Lemma~\ref{lem: burago} to \(\xi\) produces finitely many disjoint intervals whose total mean displacement is $q$ and whose total length is $t$. 

After the curve is cut at the level of the mean path, the remaining difficulty is to glue the corresponding pieces back together in the full space \(V\). It is resolved by the compactness of the quotient associated with periodicity and rearrangement invariance. 

We define the metric $d_{SS^d}: V\times V \to \R$, the equivalence relation $\sim$, and the quotient $SS^d$ as
\begin{equation*}
    d_{SS^d}(x,y):= \inf_{g\in G, z\in\Lambda} \|x-y\circ g-z\|_{L^2},\qquad
    x\sim y\Leftrightarrow d_{SS^d}(x,y)=0, \qquad
    SS^d:=V/{\sim}.
\end{equation*}
\noindent It is known \cite{Gomes-Nurbekyan-weakKAM} that $(SS^d,d_{SS^d})$ is compact. Hence there exists \(D>0\) such that for any \(x,y\in V\), one can choose \(g\in\mathcal G\) and \(z\in\Lambda\) with
\[
    \|y-(x\circ g+z)\|_{L^2}\le D.
\]
Thus, each curve piece can be rearranged and translated so that its initial point lies within a uniformly bounded \(L^2\)-distance of the endpoint of the preceding piece. Moreover, the action is preserved under this operation, since the Lagrangian is invariant under \(\Lambda\)-translations and \(\mathcal G\)-rearrangements. After saving a small amount of time and inserting short connectors between successive pieces, we obtain an admissible path for \(m(t,0,q)\), which yields the almost-superadditivity estimate. See Section~\ref{section3} for details.

\section{Almost-subadditivity and almost-superadditivity}\label{section3}
In this section, we establish the two-sided dyadic control of the mean-endpoint
metric. Almost-subadditivity follows by concatenating suitably transformed
near-minimizing curves. The reverse inequality requires the finite-dimensional
curve-cutting lemma to divide a near minimizer into two collections of pieces with
equal mean-time displacement, followed by a gluing argument in the full Hilbert
space.

\subsection{Almost-subadditivity}
We first record a connector estimate.
\begin{lem}[Uniform connector estimate]\label{lem: connector}
For every \(\tau>0\), there exists \(C_\tau>0\) such that for every \(x,y\in V\), there exist \(g\in\mathcal G\) and \(z\in\Lambda\) for which \(x\) can be connected to \(y\circ g+z\) in time \(\tau\) with action at most \(C_\tau\). 
\end{lem} 
\begin{proof} 
Fix
\[
    D>\operatorname{diam}(SS^d).
\]
Choose \(g\in\mathcal G\) and \(z\in\Lambda\) such that 
\begin{equation*} 
\|x-(y\circ g+z)\|_{L^2}\le D. 
\end{equation*} 
Let 
\begin{equation*} 
\sigma(s) = x+\frac{s}{\tau}(y\circ g+z-x), \qquad 0\le s\le\tau. 
\end{equation*} 
Then \begin{equation*}
\|\dot\sigma(s)\|_{L^2}\le \frac{D}{\tau}.
\end{equation*} 
By the upper quadratic bound on \(L\), 
\begin{equation*} 
\int_0^\tau L(\sigma,\dot\sigma)\,ds \le \tau\left( \frac12\frac{D^2}{\tau^2}+K_0 \right) = \frac{D^2}{2\tau}+K_0\tau. 
\end{equation*} 
Thus, the connector cost is uniformly bounded. 
\end{proof}

We next prove the dyadic almost-subadditivity estimate. 

\begin{prop}\label{prop: subadd}
Fix \(R>0\). There exists \(C_R>0\) such that, for all \(t>0\) and
\(y\in\mathbb R^d\) satisfying
\[
    |y|\le Rt,
\]
one has
\[
    m(2t,0,2y)
    \le
    2 m(t,0,y)+C_R.
\]
\end{prop}

\begin{proof}
We first prove the estimate for \(t\ge 2\). The case \(0<t<2\) will be absorbed
into the constant at the end.

Fix \(0<\delta<1\). Choose two near minimizers for $m(t,0,y)$
\[
    \gamma_1,\gamma_2:[0,t]\to V
\]
such that
\[
    \mathfrak m(\gamma_i(0))=0,
    \qquad
    \mathfrak m(\gamma_i(t))=y,
    \qquad i=1,2,
\]
and
\begin{equation}\label{eq: subadd_1}
    \int_0^t L(\gamma_i(s),\dot\gamma_i(s))\,ds
    \le
    m(t,0,y)+\delta,
    \qquad i=1,2.
\end{equation}
We will concatenate \(\gamma_1\) with a suitable transform of \(\gamma_2\), inserting
bounded connectors. Since the two original curves already use a total time of \(2t\),
we must first save a fixed amount of time for the connectors.

Because \(|y|\le Rt\), the straight path from $0\chi_I$ to $y\chi_I$ gives
\[
     m(t,0,y)\le C_Rt.
\]
Therefore, for each \(i=1,2\),
\[
    \int_0^t L(\gamma_i,\dot\gamma_i)\,ds
    \le
    C_Rt+\delta.
\]
Using the lower quadratic bound, we obtain
\[
    \int_0^t \|\dot\gamma_i(s)\|_{L^2}^2\,ds
    \le
    C_Rt,
    \qquad i=1,2.
\]

We now save time along \(\gamma_1\). Since \(t\ge2\), by averaging there exists an interval \([r,r+1]\subset[0,t]\) such that
\[
    \int_r^{r+1}\|\dot\gamma_1(s)\|_{L^2}^2\,ds
    \le C_R.
\]

Compress this unit interval to a length of \(1/2\). Define the modified curve
\(\widehat\gamma_1:[0,t-\frac12]\to V\) by
\begin{equation*}
\begin{cases}
    \widehat\gamma_1(s)=\gamma_1(s),
    \qquad &0\le s\le r\\
    \widehat\gamma_1(s)=\gamma_1(r+2(s-r)),
    \qquad &r\le s\le r+\frac12\\
    \widehat\gamma_1(s)=\gamma_1(s+\tfrac12),
    \qquad &r+\frac12\le s\le t-\frac12.
\end{cases}
\end{equation*}
Then \(\widehat\gamma_1\) has the same endpoints as \(\gamma_1\), but its duration is \(t-1/2\). On the compressed interval, by the upper quadratic bound on \(L\),
\begin{align}\label{eq: subadd_2}
\notag
    \int_r^{r+1/2}
    L(\widehat\gamma_1(s),\dot{\widehat\gamma}_1(s))\,ds
    &\le
    \int_r^{r+1/2}
    \left(
    \frac12\|\dot{\widehat\gamma}_1(s)\|_{L^2}^2+K_0
    \right)\,ds  \\
    &=
    \int_r^{r+1}
    \|\dot\gamma_1(\theta)\|_{L^2}^2\,d\theta
    +\frac{K_0}{2} \le C_R.
\end{align}
The original action on \([r,r+1]\) is bounded below by
\begin{equation}
    \int_r^{r+1}L(\gamma_1,\dot\gamma_1)\,ds\ge -K_0.
\end{equation}
Therefore replacing \(\gamma_1\) by \(\widehat\gamma_1\) increases the action by at
most \(C_R\), while saving exactly \(1/2\) unit of time.

Next, we transform $\gam_2$. Write
\[
    X:=\widehat\gamma_1(t-1/2),
    \qquad
    A:=\gamma_2(0),
    \qquad
    B:=\gamma_2(t).
\]
We have
\[
    \mathfrak m(X) =y,
    \qquad
    \mathfrak m(A) =0,
    \qquad
    \mathfrak m(B) = y.
\]
By Lemma~\ref{lem: connector} with $\tau = 1/4$, after choosing $g \in \mathcal{G}$ and $z \in \Lam$, we may connect $X$ to $A\circ g+z$ in time $1/4$ using the connector $\sig_1$ with uniformly bounded action.
\begin{equation}\label{eq: subadd_4}
    \int_0^{1/4} L(\sigma_1,\dot\sigma_1)\,ds\le C.
\end{equation}
Define
\[
    \widetilde\gamma_2(s):=\gamma_2(s)\circ g+z, \qquad 0\le s\le t.
\]
By the invariance of \(L\) under \(\mathcal G\)-relabelings and \(\Lambda\)-translations,
\begin{equation}\label{eq: subadd_3}
    \int L(\widetilde\gamma_2,\dot{\widetilde\gamma}_2)\,ds
    =
    \int L(\gamma_2,\dot\gamma_2)\,ds.
\end{equation}
Let
\[
    \widetilde A:=\widetilde\gam_2 (0) = A\circ g+z,
    \qquad
    \widetilde B:=\widetilde\gam_2 (t) = B\circ g+z.
\]
Then
\begin{equation}\label{eq: endpoint displacement}
    \|X-\widetilde A\|_{L^2}\le D.
\end{equation}
Moreover, the mean displacement of the second path is unchanged because the same translation $z$ is added to both endpoints.
\begin{equation}\label{eq: mean displacement}
    \mathfrak m(\widetilde B) - \mathfrak m(\widetilde A)
    =
    \mathfrak m(B) - \mathfrak m(A)
    =
    y.
\end{equation}

After following \(\widehat\gamma_1\), $\sig_1$, and \(\widetilde\gamma_2\), the endpoint
is \(\widetilde B\). Its mean need not be exactly \(2y\). Define the mean error
\[
    e:=\mathfrak m(\widetilde B) - 2y.
\]
From \eqref{eq: endpoint displacement} and \eqref{eq: mean displacement}, we have
\begin{equation}\label{eq: endpoint errorbound}
    |e|
    =
    |\mathfrak m(\widetilde A) - \mathfrak m(X) |
    \le
    \|\widetilde A-X\|_{L^2}
    \le D.
\end{equation}
Thus the mean error is uniformly bounded.

We correct the endpoint by adding a final straight connector $\sig_2$ from \(\widetilde B\) to $\widetilde B-e\chi_I$ using the remaining time $1/4$:
\begin{equation*}
    \sig_2 (s) = \widetilde B - 4se\chi_I,\qquad 0\le s\le \frac14
\end{equation*}
The final point has mean $\mathfrak m(\widetilde B)-e = 2y$. From \eqref{eq: endpoint errorbound},
\begin{equation}\label{eq: subadd_5}
    \int_0^{1/4} L(\sigma_2,\dot\sigma_2)\,ds\le C.
\end{equation}

We concatenate the curves $\widehat\gam_1$, $\sig_1$, $\widetilde\gam_2$, and $\sig_2$. The resulting curve has initial mean \(0\), final mean \(2y\), and total duration $2t$. Therefore, it is an admissible curve for \( m(2t,0,2y)\). Combining \eqref{eq: subadd_1}-\eqref{eq: subadd_3} and \eqref{eq: subadd_5}, we obtain
\[
     m(2t,0,2y)
    \le
    2 m(t,0,y)+C_R+2\delta.
\]
Letting \(\delta\to0\), we obtain
\[
     m(2t,0,2y)
    \le
    2 m(t,0,y)+C_R
\]
for all \(t\ge2\) and \(|y|\le Rt\).

It remains to treat \(0<t<2\). In this range, \(|y|\le Rt\) implies
\[
     m(2t,0,2y)\le C_R,
\]
by the straight mean path. On the other hand, the lower quadratic bound gives
\[
     m(t,0,y)\ge -K_0t\ge -2K_0.
\]
Thus, after increasing \(C_R\), we also have
\[
     m(2t,0,2y)
    \le
    2 m(t,0,y)+C_R
\]
for \(0<t<2\). The proof is complete.
\end{proof}

\begin{rem}[General almost-subadditivity]\label{rem:general-subadditivity}
The argument in Proposition~\ref{prop: subadd} extends to curves with different durations and mean displacements. In particular, for every \(R>0\), there exists \(C_R>0\) such that the following estimates hold.

If \(t>0\), \(a,q\in\mathbb R^d\), and \(\rho,\sigma>0\) satisfy
\[
    |q-a|\le Rt,
    \qquad
    \rho t\ge 1,
    \qquad
    \sigma t\ge 1,
\]
then
\begin{equation}\label{eq: general subadd}
    m\bigl((\rho+\sigma)t,(\rho+\sigma)a,(\rho+\sigma)q\bigr)
    \le
    m(\rho t,\rho a,\rho q)
    +m(\sigma t,\sigma a,\sigma q)
    +C_R.
\end{equation}
These estimates follow by repeating the concatenation and curve-rearrangement arguments in Proposition~\ref{prop: subadd}. The only additional modification is the insertion of a uniformly bounded initial connector. Consequently, for fixed \(t,a,q\), the function
\[
    \rho\longmapsto m(\rho t,\rho a,\rho q), \qquad \rho > 0
\]
is almost-subadditive. This is the estimate used in Section~\ref{section4} to construct the homogenized mean metric by Fekete's lemma.

Also, if \(t,s\ge 1\) and \(a,b,c\in\mathbb R^d\) satisfy
\[
    |b-a|\le Rt,
    \qquad
    |c-b|\le Rs,
\]
then
\begin{equation}\label{eq: general subadd_2}
    m(t+s,a,c)
    \le
    m(t,a,b)+m(s,b,c)+C_R.
\end{equation}

In the finite-dimensional case considered in~\cite{Tran-Yu-optimal}, \eqref{eq: general subadd_2} holds without $C_R$. However, in the present infinite-dimensional setting, we need an extra constant because the metric prescribes only the mean of the endpoints. Thus, the endpoint of a near minimizer for \(m(t,a,b)\) need not coincide with the initial point of a near minimizer for \(m(s,b,c)\), even though both have mean \(b\). In order to construct the admissible curve for $m(t+s, a,c)$, we therefore transform the second curve and insert a uniformly bounded connector, and finally correct the resulting bounded terminal mean error. These modifications produce the additional constant $C_R$.
\end{rem}

\subsection{Almost-superadditivity}
We next prove the reverse inequality.

\begin{prop}\label{prop: superadd}
Fix \(R>0\). There exists \(C_R>0\) such that, for all \(t>0\) and
\(y\in\mathbb R^d\) satisfying
\[
    |y|\le Rt,
\]
one has
\[
    2m(t,0,y)\le m(2t,0,2y)+C_R.
\]
\end{prop}

\begin{proof}
We first prove the estimate for $t\ge d+4$. The bounded-time case \(0<t< d+4\) will be absorbed into the constant at the end.

Fix \(0<\delta<1\). Choose an absolutely continuous curve
\[
    \gamma:[0,2t]\to V
\]
such that
\[
    \mathfrak m(\gamma(0))=0,
    \qquad
    \mathfrak m(\gamma(2t)) = 2y,
\]
and
\begin{equation}
    \int_0^{2t}L(\gamma(s),\dot\gamma(s))\,ds
    \le
    m(2t,0,2y)+\delta.
\end{equation}
Because \(|y|\le Rt\), the straight path from $0\chi_I$ to $2y\chi_I$ gives
\[
    m(2t,0,2y) \le C_Rt.
\]
Using the lower quadratic bound on $L$, we obtain
\[
    \int_0^{2t}\|\dot\gamma(s)\|_{L^2}^2\,ds
    \le C_Rt.
\]

Define the finite-dimensional mean-time path
\[
    \xi(s):=(\mathfrak m(\gamma(s)),s),
    \qquad 0\le s\le 2t.
\]
Then
\[
    \xi(2t)-\xi(0)=(2y,2t).
\]
By applying the topological lemma to \(\xi\), there exist pairwise disjoint intervals
\[
    [a_i,b_i]\subset[0,2t],
    \quad i=1,\dots,k, \quad \text{with} \quad k \le \frac{d+2}{2},
\]
such that
\[
    \sum_{i=1}^k
    \bigl(\xi(b_i)-\xi(a_i)\bigr)
    =
    \frac{\xi(2t)-\xi(0)}{2}
    =
    (y,t).
\]
Equivalently,
\[
    \sum_{i=1}^k
    \bigl(\mathfrak m(\gamma(b_i))-\mathfrak m(\gamma(a_i))\bigr)
    =
    y, \quad \text{and} \quad \sum_{i=1}^k(b_i-a_i)=t.
\]
Let $A_{\rm sel}$ denote the total action along the selected pieces.
\[
    A_{\rm sel}
    :=
    \sum_{i=1}^k
    \int_{a_i}^{b_i}L(\gamma(s),\dot\gamma(s))\,ds.
\]
The selected pieces have total time \(t\), total mean displacement \(y\chi_I\), and at
most \((d+2)/2\) components. Their total kinetic energy satisfies
\[
    \sum_{i=1}^k
    \int_{a_i}^{b_i}\|\dot\gamma(s)\|_{L^2}^2\,ds
    \le C_Rt.
\]

We now save \(1/2\) unit of time from the selected pieces. Since the selected intervals have total length \(t\) and there are at most \((d+2)/2\) of them, one selected interval, say \([a_{i_0},b_{i_0}]\), has length at least $2t/(d+2) \ge 2$. Hence it contains a unit subinterval. By averaging over \([a_{i_0},b_{i_0}]\) and using the total kinetic-energy bound above, we may choose \([r,r+1]\subset[a_{i_0},b_{i_0}]\)
so that
\[
\int_r^{r+1}\|\dot\gamma(s)\|_{L^2}^2\,ds\le C_R.
\]
We compress this unit interval to a length of \(1/2\). 
Then this compression increases the selected action by at most \(C_R\), as in the subadditivity proof. We denote
the compressed selected pieces by
\[
    \eta_i:[0,\ell_i]\to V,
    \qquad i=1,\dots,k.
\]
They satisfy
\begin{equation} \label{eq: sum of means}
    \sum_{i=1}^k
    \left(
    \mathfrak m(\eta_i(\ell_i))-\mathfrak m(\eta_i(0))
    \right)
    =
    y, \qquad \sum_{i=1}^k\ell_i=t-\frac12,
\end{equation}
and
\begin{equation}\label{eq: total action of compressed}
    \sum_{i=1}^k
    \int_0^{\ell_i}L(\eta_i,\dot\eta_i)\,ds
    \le
    A_{\rm sel}+C_R.
\end{equation}

We next glue these compressed selected pieces into one admissible curve for \(m(t,0,y)\) in the spirit of Lemma~\ref{lem: connector}.
For the first piece, choose \(g_1\in\mathcal G\) and \(z_1\in\Lambda\) such that
\[
    \|\eta_1(0)\circ g_1+z_1\|_{L^2}\le D.
\]
Set
\[
    \tilde \eta_1(s):=\eta_1(s)\circ g_1+z_1,
    \qquad
    P_1:= \tilde \eta_1(0),
    \qquad
    Q_1:=\tilde \eta_1(\ell_1).
\]
Suppose that the transformed pieces have been constructed up to piece \(j-1\), and
let \(Q_{j-1}\) be the endpoint of the \((j-1)\)-st transformed piece. Choose
\(g_j\in\mathcal G\) and \(z_j\in\Lambda\) such that
\[
    \|Q_{j-1}-(\eta_j(0)\circ g_j+z_j)\|_{L^2}\le D.
\]
Define
\[
    \tilde \eta_j(s):=\eta_j(s)\circ g_j+z_j,
    \qquad
    P_j:=\tilde \eta_j(0),
    \qquad
    Q_j:=\tilde \eta_j(\ell_j).
\]
The invariance of \(L\) under relabelings and lattice translations gives
\begin{equation}\label{eq: unchanging action}
    \int_0^{\ell_j}L(\tilde \eta_j,\dot{\tilde{  \eta}}_j)\,ds
    =
    \int_0^{\ell_j}L(\eta_j,\dot\eta_j)\,ds.
\end{equation}
Moreover, the mean displacement of each piece is unchanged:
\begin{equation}\label{eq: unchanged mean displacements}
    \mathfrak m( Q_j) - \mathfrak m( P_j)
    =
    \mathfrak m(\eta_j(\ell_j)) - \mathfrak m(\eta_j(0)).
\end{equation}

We now connect $Q_{j-1}$ and $P_j$ using a straight path for $j=1,\dots,k$ with $Q_0 = 0\chi_I$. Here, each connector uses time $\frac{1}{2(k+1)}$. More precisely,
\begin{align*}
    \sig_j(s) = Q_{j-1} + 2s(k+1)(P_{j} - Q_{j-1}), \qquad 0 \le s \le \frac{1}{2(k+1)}
\end{align*}
for $j = 1,\dots, k$. By construction, these connector lengths are bounded by \(D\):
\begin{equation}\label{eq: bd for endpoints}
    \|P_{j}-Q_{j-1}\|_{L^2}\le D,
\end{equation}
for $j = 1,\dots, k$. After inserting $\sig_j$ before $\widetilde \eta_j$ for each $j = 1,\dots k$, the endpoint of the resulting curve is \(Q_k\). Its mean may not equal \(y\), so we estimate the mean error. Define
\[
    c_j:=\mathfrak m(P_{j}) - \mathfrak m(Q_{j-1}),
    \qquad j=1,\dots,k.
\]
Then by \eqref{eq: bd for endpoints},
\[
    |c_j| \le D, \qquad j = 1,\dots, k.
\]
We have the telescoping identity
\[
    \mathfrak m(Q_k)
    =
    c_1
    +(\mathfrak m(Q_1) - \mathfrak m(P_1))
    +c_2
    +(\mathfrak m(Q_2) - \mathfrak m(P_2))
    +\cdots
    +c_{k}
    +(\mathfrak m(Q_k) - \mathfrak m(P_k)).
\]
Using \eqref{eq: sum of means} and \eqref{eq: unchanged mean displacements}, we get
\[
    \mathfrak m(Q_k)- y =c_1+c_2+\cdots+c_{k}.
\]
Therefore
\begin{equation}\label{eq: bd for error}
    |\mathfrak m(Q_k)-y |\le kD.
\end{equation}
Set
\[
    e:= \mathfrak m(Q_k) - y.
\]
We add a final connector from \(Q_k\) to $Q_k-e\chi_I$ using the remaining time $\frac{1}{2(k+1)}$:
\begin{equation*}
    \sig_{k+1}(s) = Q_k - 2s(k+1)e\chi_I, \qquad 0 \le s \le \frac{1}{2(k+1)}.
\end{equation*}
The final point has mean $y$. By \eqref{eq: bd for error}, the length of the final connector is bounded by $kD$.
Together with \eqref{eq: bd for endpoints}, the lengths of connectors $\{\sig_j\}_{j = 1,\dots k+1}$ are bounded by a constant depending only on \(d\) and \(D\).

There are \(k+1\) connectors in total and each connector receives time $\frac{1}{2(k+1)}$.
Because \(k\le(d+2)/2\), this time is bounded below by a positive constant depending only on \(d\). Hence, the quadratic upper bound on \(L\) gives a uniform bound on the total connector action: 
\begin{equation} \label{eq: connector action}
    \sum_{j=1}^{k+1} \int_0^{\frac{1}{2(k+1)}} L(\sig_j, \dot\sig_j) \, ds \leq C.
\end{equation}
Therefore, concatenating $\{\widetilde\eta_j\}_{j = 1,\dots k}$ and $\{\sig_j\}_{j = 1,\dots k+1}$ produces an admissible curve for \(m(t,0,y)\), and combining \eqref{eq: total action of compressed}, \eqref{eq: unchanging action}, and \eqref{eq: connector action} gives
\begin{equation}\label{eq: super_sel}
    m(t,0,y)\le A_{\rm sel}+C_R.
\end{equation}

We repeat the same construction for the complementary intervals. Let
\[
    [\alpha_j,\beta_j],
    \qquad j=1,\dots,\ell,
\]
be the nontrivial connected components of
\[
    [0,2t]\setminus \bigcup_{i=1}^k(a_i,b_i).
\]
Then $\ell\le k+1\le (d+4)/2$.

Since the whole curve carries mean-time displacement \((2y,2t)\), while the selected intervals carry \((y,t)\), the complementary intervals also carry \((y,t)\). Then, one of them has length at least
\[
    \frac{t}{\ell}\geq\frac{2t}{d+4}\geq2.
\]
Hence the same averaging and time-saving argument used for the selected intervals applies to the complementary intervals. 

Let
\[
    A_{\rm comp}
    :=
    \sum_{j=1}^{\ell} \int_{\alpha_j}^{\beta_j}L(\gamma(s),\dot\gamma(s))\,ds.
\]
By the same construction, we obtain
\begin{equation}\label{eq: super_comp}
    m(t,0,y)\le A_{\rm comp}+C_R,
\end{equation}
for \(t\ge d+4\). 

Adding \eqref{eq: super_sel} and \eqref{eq: super_comp}, we obtain
\[
\begin{aligned}
    2m(t,0,y)
    &\le
    A_{\rm sel}+A_{\rm comp}+C_R  \\
    &=
    \int_0^{2t}L(\gamma(s),\dot\gamma(s))\,ds+C_R \\
    &\le
    m(2t,0,2y)+C_R+\delta.
\end{aligned}
\]
Letting \(\delta\to0\), we get
\[
    2m(t,0,y)\le m(2t,0,2y)+C_R
\]
for all \(t\ge d+4\).

Finally, suppose \(0<t<d+4\). Since \(|y|\le Rt\), the straight mean path gives
\[
    m(t,0,y)\le C_R.
\]
The lower quadratic bound gives
\[
    m(2t,0,2y)\ge -2K_0(d+4).
\]
After increasing \(C_R\), we obtain
\[
    2m(t,0,y)\le m(2t,0,2y)+C_R
\]
also in this bounded-time range. The proof is complete.
\end{proof}

Figure~\ref{fig:curve-cutting} summarizes the curve-cutting and gluing construction underlying the almost-superadditivity estimate.

\begin{figure}[!htbp]
    \centering
    \includegraphics[width=\textwidth]
    {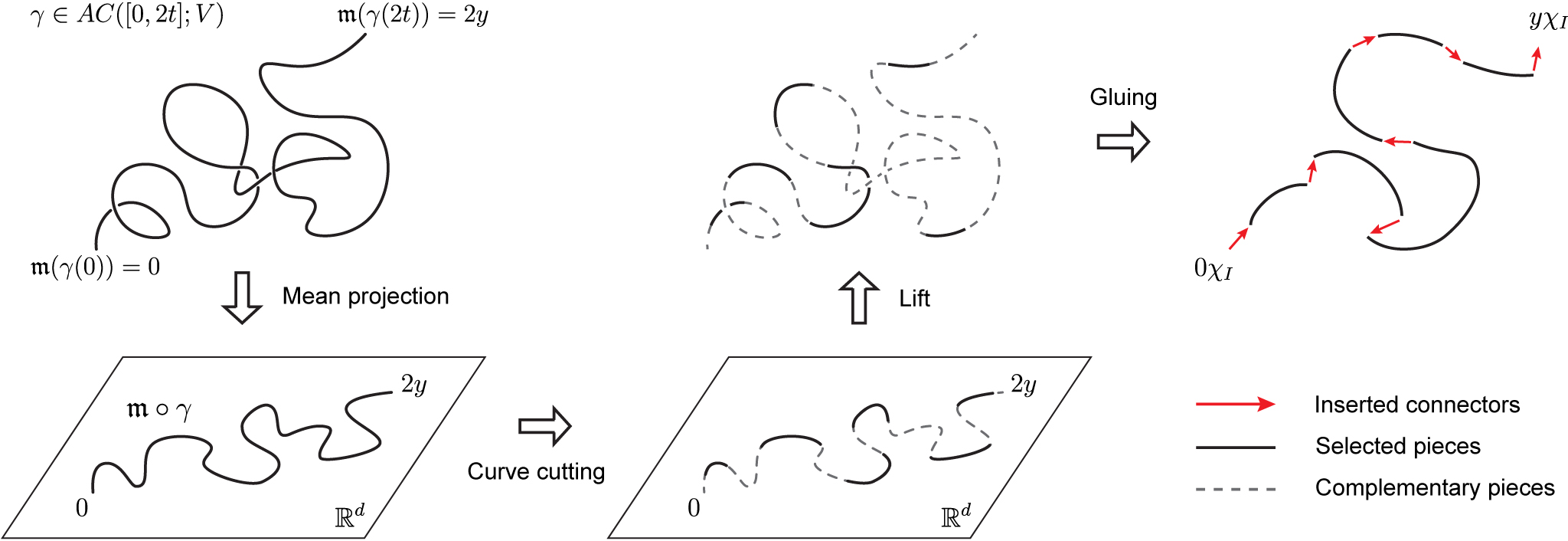}
    \caption{Illustration of the curve-cutting and gluing construction in the proof of Proposition~\ref{prop: superadd}. A near-minimizing curve for $m(2t,0,2y)$ is projected onto the mean-time path. Lemma~\ref{lem: burago} selects intervals with total mean-time displacement $(y,t)$. The corresponding pieces of $\gamma$ are lifted back to $V$, rearranged and translated, and joined by uniformly bounded connectors to produce an admissible path for $m(t,0,y)$.}
    \label{fig:curve-cutting}
\end{figure}

The same argument also gives the shifted form of the dyadic estimate.

\begin{cor}\label{cor: general super}
    For every \(R>0\), there exists \(C_R>0\) such that
    \[
        2m(t,a,q) \le m(2t,2a,2q) + C_R
    \]
    whenever \(t>0\) and \(|q-a|\le Rt\).
\end{cor}

To prove Corollary~\ref{cor: general super}, we apply Lemma~\ref{lem: burago} to the same mean-time path. Then, the selected and complementary families each have duration $t$ and mean displacement $q-a$. After repeating the gluing argument with initial mean $a$ and terminal mean $q$, we obtain the result.

Combining~\eqref{eq: general subadd} with Corollary~\ref{cor: general super}, we obtain the two-sided dyadic estimate
\begin{equation}\label{eq: s3_22}
    \left|
    m(2t,2a,2q)-2m(t,a,q)
    \right|
    \le C_R
\end{equation}
whenever \(|q-a|\le Rt\). In the next section, the general almost-subadditivity estimate will provide the existence of the large-scale limit, while this two-sided dyadic bound will yield the uniform \(O(1)\) error.

\section{The homogenized metric and the effective Lagrangian} \label{section4}
\subsection{The homogenized mean metric}
We now construct the homogenized metric from the large-scale behavior of the mean-endpoint metric. The general almost-subadditivity estimate gives the existence of the large-scale limit through Fekete's lemma, while the two-sided dyadic estimate yields a uniform bound on the error between the mean-endpoint metric and the homogenized metric.

\begin{prop}\label{prop: existance mbar}
Fix \(R>0\). There exists \(C_R>0\) such that, for every \(t>0\) and \(a,q\in\mathbb R^d\) satisfying $|q-a|\le Rt$, the limit
\begin{equation}\label{eq: s5_0}
    \overline m(t,a,q)
    :=
    \lim_{\rho\to\infty}
    \frac{1}{\rho}
    m(\rho t,\rho a,\rho q)
\end{equation}
exists. Moreover,
\begin{equation}\label{eq: s5_0.5}
    \left|
    m(t,a,q)-\overline m(t,a,q)
    \right|
    \le C_R.
\end{equation}
\end{prop}

\begin{proof}
Fix \(t>0\) and \(a,q\in\mathbb R^d\) with \(|q-a|\le Rt\), and define
\[
    F(\rho):=m(\rho t,\rho a,\rho q),
    \qquad \rho>0.
\]
The general almost-subadditivity estimate from Section 3 gives
\begin{equation}\label{eq: s5_1}
    F(\rho+\sigma)
    \le
    F(\rho)+F(\sigma)+C_R
\end{equation}
for all sufficiently large \(\rho,\sigma>0\). The quadratic bounds on \(L\) also give
\begin{equation}\label{eq: s5_2}
    -C_{t,R} \rho\le F(\rho)\le C_{t,R} \rho.
\end{equation}
Together with \eqref{eq: s5_1} and \eqref{eq: s5_2}, the continuous almost-subadditive version of Fekete's lemma yields the existence of the finite limit $\lim_{\rho \to \infty} F(\rho)/\rho$. This proves the existence of the homogenized mean-endpoint metric $\overline m$.

We next prove the bounded-error estimate.
Set
\[
    A_n
    :=
    2^{-n}m(2^nt,2^na,2^nq), \qquad n\in \N.
\]
By Propositions~\ref{prop: subadd} and \ref{prop: superadd}, together with~\eqref{eq: s3_22}, we obtain
\[
    |A_{n+1}-A_n|
    \le
    2^{-(n+1)}C_R, \qquad n\in\N.
\]
Consequently,
\begin{equation}\label{eq: s5_3}
    |A_N-A_0|
    \le
    \sum_{n=0}^{N-1}|A_{n+1}-A_n| 
    \le
    C_R\sum_{n=0}^{N-1}2^{-(n+1)}
    \le C_R.
\end{equation}
Since the limit~\eqref{eq: s5_0} exists, its dyadic subsequence converges to the same limit:
\[
    \lim_{N\to\infty}A_N
    =
    \overline m(t,a,q).
\]
Letting \(N\to\infty\) in~\eqref{eq: s5_3}, we conclude that
\[
    \left|
    m(t,a,q)-\overline m(t,a,q)
    \right|
    \le C_R.
\]
This proves~\eqref{eq: s5_0.5}.
\end{proof}

We next establish the structural properties of the homogenized metric. These properties reduce the general mean-endpoint problem to the unit-time, zero-initial-mean function \(q\mapsto\overline m(1,0,q)\), which will be identified with the effective Lagrangian.

\begin{lem}\label{lem: mbarproperty}
For every \(t>0\), \(a,q,c\in\mathbb R^d\), and \(\lambda>0\), the homogenized
mean metric satisfies
\begin{equation}\label{eq: s5_4}
    \overline m(\lambda t,\lambda a,\lambda q)
    =
    \lambda\overline m(t,a,q)
\end{equation}
and
\begin{equation}\label{eq: s5_5}
    \overline m(t,a+c,q+c)
    =
    \overline m(t,a,q).
\end{equation}
\end{lem}

\begin{proof}
We first prove positive homogeneity. By construction,
\[
\begin{aligned}
    \overline m(\lambda t,\lambda a,\lambda q)
    &=
    \lim_{\rho\to\infty}
    \frac{1}{\rho}
    m(\rho\lambda t,\rho\lambda a,\rho\lambda q) \\
    &=
    \lambda
    \lim_{\rho\to\infty}
    \frac{1}{\rho\lambda}
    m(\rho\lambda t,\rho\lambda a,\rho\lambda q) = \lambda\overline m(t,a,q).
\end{aligned}
\]

We next record a bounded endpoint-perturbation estimate. For every \(R,A>0\), there exists \(C_{R,A}>0\) such that
\begin{equation}\label{eq: s5_6}
    \left|
    m(t,a+\alpha,q+\beta)-m(t,a,q)
    \right|
    \le C_{R,A}
\end{equation}
whenever
\[
    t\ge 2,\qquad |q-a|\le Rt,
    \qquad |\alpha|+|\beta|\le A.
\]
The proof is analogous to Proposition~\ref{prop: subadd}. Take a near minimizer \(\gamma:[0,t]\to V\) for \(m(t,a,q)\) and save \(1/2\) unit of time while increasing the action by at most \(C_R\). We then use the saved time to insert a connector from \(\gam(0)+\alpha\chi_I\) to \(\gam(0)\) and another connector from \(\gam(t)\) to \(\gam(t)+\beta\chi_I\). The lengths of these connectors are bounded by \(A\). The resulting curve is admissible for \(m(t,a+\alpha,q+\beta)\), and therefore
\[
    m(t,a+\alpha,q+\beta)
    \le
    m(t,a,q)+C_{R,A}.
\]
Applying the same argument in the reverse direction proves~\eqref{eq: s5_6}.

Now fix \(c\in\mathbb R^d\). For every \(\rho>0\), choose
\(z_\rho\in\mathbb Z^d\) such that
\begin{equation*}\label{eq: s5_7}
    |\rho c-z_\rho|\le C_d
\end{equation*}
for a dimensional constant $C_d$. The periodicity of \(L\) implies
\[
    m(\rho t,\rho(a+c),\rho(q+c))
    =
    m(\rho t,\rho a+\rho c-z_\rho,\rho q+\rho c-z_\rho).
\]
Hence, for all sufficiently large \(\rho\), the endpoint-perturbation estimate gives
\begin{equation}\label{eq: s5_8}
    \big|
    m(\rho t,\rho(a+c),\rho (q+c))
    -
    m(\rho t,\rho a,\rho q)
    \big|
    \le C_R.
\end{equation}
Dividing~\eqref{eq: s5_8} by \(\rho\) and sending \(\rho\to\infty\), we obtain
\[
    \overline m(t,a+c,q+c)
    =
    \overline m(t,a,q).
\]
This proves \eqref{eq: s5_5}.
\end{proof}

We conclude with the quantitative metric estimate used in the proof of the optimal
homogenization rate.

\begin{lem}\label{lem: bddestimate_mbar}
Fix \(R>0\). There exists \(C_R>0\) such that, for every \(t>0\),
\(a,q\in\mathbb R^d\), and \(\varepsilon > 0\) satisfying
\[
    |q-a|\le Rt,
\]
we have
\[
    \left|
    \varepsilon
    m\left(
    \frac{t}{\varepsilon},
    \frac{a}{\varepsilon},
    \frac{q}{\varepsilon}
    \right)
    -
    \overline m(t,a,q)
    \right|
    \le C_R\varepsilon.
\]
\end{lem}

\begin{proof}
Lemma~\ref{lem: mbarproperty} implies
\[
    \overline m(t,a,q)
    =
    t\overline m\left(1,0,\frac{q-a}{t}\right).
\]
Then Proposition~\ref{prop: existance mbar} proves the lemma.
\end{proof}

We have now constructed the homogenized metric and shown that it approximates the microscopic mean-endpoint metric up to a uniformly bounded error. It remains to identify the homogenized metric with the effective Lagrangian arising from the cell problem.

\subsection{Identification of the homogenized mean metric}
In this section, we identify \(\overline m\) with the effective Lagrangian. The identification is obtained through convex duality. We define the Legendre dual of the unit-time homogenized metric and compare it with the effective Hamiltonian. Recall that the effective Hamiltonian is only defined on the subspace of mean configurations. 
% Write
% \[ 
% \bar H_e(p)=\bar H(p\chi_I), \qquad p\in\mathbb R^d, 
% \] 
% and let

Let
\[ 
\overline L(q):= \sup_{p \in \R^d} \{p\cdot q - \overline H (p)\}, \qquad q\in\R^d
\] 
be the effective Lagrangian, which is the Legendre transform of the effective Hamiltonian.

\begin{prop}\label{prop: mbar identification}
For every \(t>0\) and \(a,q\in\mathbb R^d\), we have 
\[ 
\overline m(t,a,q) = t\overline L\left(\frac{q-a}{t}\right). 
\] 
\end{prop}
\begin{proof}
We consider the Legendre transform of the homogenized mean metric. Define
\[
    \overline H_m(p)
    :=
    \sup_{q\in\mathbb R^d}
    \left\{
    p\cdot q- \overline m(1,0,q)
    \right\},
    \qquad p\in\mathbb R^d.
\]
We will prove that
\[
    \overline H_m(p)=\overline H(p)
\]
for every $p \in \R^d$ and the function  
\[
    q\mapsto \overline m(1,0,q), \qquad q \in \R^d
\]
is convex and lower semicontinuous. Once these facts are established, the conclusion follows from the Fenchel--Moreau theorem. More precisely,
\[
    \overline m(1,0,q)
    =
    \overline H_m^*(q)
    =
    \overline H^*(q)
    =
    \overline L(q), \qquad \text{for every } q \in \R^d.
\]
Using Lemma~\ref{lem: mbarproperty}, we obtain
\[
    \overline m(t,a,q)
    =
    t\overline m\left(1,0,\frac{q-a}{t}\right)
    =
    t\overline L\left(\frac{q-a}{t}\right)
\]
for every $t>0$ and $a,q \in\R^d$. 

\noindent\textbf{Step 1.} Fix \(p\in\mathbb R^d\). Consider the tilted Hamilton--Jacobi equation
\begin{equation}\label{eq: tilted HJ}
\begin{cases}
    w_t+H(x,p\chi_I+Dw)=0, \qquad &\text{in } V\times (0,\infty)\\
    w(x,0)=0, \qquad &\text{on } V.
\end{cases}
\end{equation}
Equivalently, set
\[
    H_p(x,r):=H(x,p\chi_I+r),
    \qquad r\in V.
\]
The Lagrangian associated with \(H_p\) is
\[
\begin{aligned}
    L_p(x,\xi)
    = \sup_{r\in V}
    \left\{\langle r,\xi\rangle_{L^2}
    - H(x,p\chi_I+r)\right\} 
    = L(x,\xi)-\langle p\chi_I,\xi\rangle_{L^2}.
\end{aligned}
\]
Therefore, by the optimal-control representation,
\begin{align*}
    w(x,t)
    &=
    \inf_{\substack{\gamma\in AC([0,t];V)\\ \gamma(t)=x}}
    \int_0^t
    L(\gamma(s),\dot\gamma(s))
    -
    \langle p\chi_I, \dot\gamma(s) \rangle_{L^2}
    \,ds \\
    &=
    \inf_{\substack{\gamma\in AC([0,t];V)\\ \gamma(t)=x}}
    \left\{
    \int_0^t L(\gamma(s),\dot\gamma(s))\,ds
    -
    p \cdot \left(\mathfrak m(x) - \mathfrak m(\gamma(0))\right)
    \right\}.
\end{align*}

For \(q\in\mathbb R^d\), define the mean-restricted value function
\[
    W_p(q,t)
    :=
    \inf_{\mathfrak m(x)=q} w(x,t).
\]
Grouping admissible curves according to their initial mean \(a \chi_I\), we get
\[
    W_p(q,t)
    =
    \inf_{a\in\mathbb R^d}
    \left\{
    m(t,a,q) - p\cdot(q - a)
    \right\}.
\]
In particular, for \(q=0\),
\[
    W_p(0,t)
    =
    \inf_{a\in\mathbb R^d}
    \left\{
    m(t,a,0)+p\cdot a
    \right\}.
\]
Writing \(a=-tv\), we obtain
\begin{equation}\label{eq: s6_Wp}
    \frac1t W_p(0,t)
    =
    \inf_{v\in\mathbb R^d}
    \left\{
    \frac1t m(t,-tv,0)-p\cdot v
    \right\}.
\end{equation}

\noindent\textbf{Step 2.}
We claim that
\begin{equation}\label{eq: prop6.1claim}
    \lim_{t\to\infty}\frac1t W_p(0,t)
    =
    -\overline{H}_m(p).
\end{equation}
Indeed, for $v$ in a bounded subset of $\R^d$, by Proposition~\ref{prop: existance mbar},
\[
    m(t,-tv,0)
    =
    \overline{m}(t,-tv,0) + O(1)
\]
uniformly in time. Then, by Lemma~\ref{lem: mbarproperty},
\begin{equation}\label{eq: s6_bd}
    \frac1t m(t,-tv,0) = \frac1t \overline m(t,-tv,0) + O\left(\frac1t\right) = \overline m (1,0,v) + O\left(\frac1t\right).
\end{equation}

It remains to show that the infimum in \eqref{eq: s6_Wp} may be restricted to a fixed bounded set of \(v\)'s. For every admissible curve $\gam$ for $m(t,-tv,0)$, Jensen's inequality gives
\[
    \int_0^t \|\dot\gamma(s)\|_{L^2}^2\,ds
    \ge
    t|v|^2.
\]
Therefore,
\[
    \frac1t m(t,-tv,0)-p\cdot v
    \ge
    \frac12|v|^2-p\cdot v-K_0.
\]
The right-hand side tends to \(+\infty\) as \(|v|\to\infty\), uniformly in \(t\). On the other hand, taking $v=0$ and using the constant path at $0\chi_I$ gives
\[
    \frac1t m(t,0,0)\leq K_0
\]
uniformly in $t$. Thus the infima are confined to a ball depending only on \(p\). We may therefore pass to the limit through the infimum and use \eqref{eq: s6_bd} to obtain
\[
\begin{aligned}
    \lim_{t\to\infty}\frac1t W_p(0,t)
    =
    \inf_{v\in\mathbb R^d}
    \left\{\overline{m}(1,0,v)-p\cdot v \right\} = -\overline{H}_m(p).
\end{aligned}
\]
This proves the claim \eqref{eq: prop6.1claim}.

\noindent\textbf{Step 3.}
We now compute the same large-time limit using the effective Hamiltonian. For each \(p\in\R^d\), the cell problem
\[
    H(x,p\chi_I+Dv_p(x))=\overline{H}(p), \qquad \text{in }V
\]
admits a bounded corrector \(v_p\). See Proposition~3.2 in~\cite{park2026}. Then
\[
    v_p(x)-\overline{H}(p)t
\]
solves the tilted equation \eqref{eq: tilted HJ} with initial data \(v_p\). The comparison principle gives
\[
    \left|
    w(x,t)+t\overline{H}(p)
    \right|
    \le
    2 \|v_p\|_{L^\infty}
\]
for all \(x\in V\) and \(t>0\). This also implies
\[
    \left|
    W_p(0,t)+t\overline{H}(p)
    \right|
    \le
    2 \|v_p\|_{L^\infty}.
\]
Dividing by \(t\) and sending \(t\to\infty\), we get
\begin{equation}\label{eq: s6_largetimelimit}
    \lim_{t\to\infty}\frac1t W_p(0,t)
    =
    -\overline{H}(p).
\end{equation}

Combining the two expressions \eqref{eq: prop6.1claim} and \eqref{eq: s6_largetimelimit}  gives
\[
    -\overline{H}_m(p)
    =
    -\overline{H}(p).
\]
Hence
\[
    \overline{H}_m(p)=\overline{H}(p)
    \qquad\text{for every }p\in\mathbb R^d.
\]

\noindent\textbf{Step 4.}
It remains only to justify the convex duality step. Passing to the limit $\rho\to\infty$ in the quadratic bounds on \(m\)
\[
\frac12|q|^2-K_0 \leq \frac1\rho m(\rho,0,\rho q) \leq \frac12|q|^2+K_0, \qquad q \in \R^d
\]
gives
\[
\frac12|q|^2-K_0 \leq \overline m(1,0,q) \leq \frac12|q|^2+K_0, \qquad q \in \R^d.
\]

In particular, \(q\mapsto \overline{m}(1,0,q)\) is finite on \(\mathbb R^d\).

We next show that this function is convex. Let \(q_1,q_2\in\mathbb R^d\) and
\(\theta\in(0,1)\). Set
\[
    q_\theta:=\theta q_1+(1-\theta)q_2.
\]
For sufficiently large $t>0$, the general almost-subadditivity estimate~\eqref{eq: general subadd_2} gives
\[
    m(t,0,tq_\theta)
    \le
    m(\theta t,0,\theta tq_1)
    +
    m((1-\theta)t,\theta tq_1,tq_\theta)
    +
    C.
\]
Therefore, after passing to the homogenized metric and using Lemma~\ref{lem: mbarproperty}, we get
\[
\begin{aligned}
    \overline{m}(1,0,q_\theta)
    &\le
    \theta \overline{m}(1,0,q_1)
    +
    (1-\theta)\overline{m}(1,0,q_2).
\end{aligned}
\]
Thus \(q\mapsto \overline{m}(1,0,q)\) is convex. Since it is finite and convex on
\(\mathbb R^d\), it is continuous, and hence lower semicontinuous. This completes the proof.
\end{proof}

\section{Conclusion: the optimal convergence rate}\label{section5} 
\begin{proof}[Proof of Theorem~1.1]
Recall that for every $q \in \R^d$, we define the mean-restricted value function:
\begin{equation}\label{eq: s7_1}
    U^\varepsilon(q,t)
    :=
    \inf_{\mathfrak m(x) = q} u^\varepsilon(x,t) = \inf_{a\in\mathbb R^d}
    \left\{
    \widetilde u_0(a)
    +
    \varepsilon
    m\left(
    \frac{t}{\varepsilon},
    \frac{a}{\varepsilon},
    \frac{q}{\varepsilon}
    \right)
    \right\}.
\end{equation}
Let
\[
    K:=\|Du_0\|_{L^\infty(V)}.
\]
We first show that the infima in~\eqref{eq: s7_1} are restricted to a bounded ball.
By the quadratic lower bound on \(L\) and Jensen's inequality,
\[
    \varepsilon
    m\left(
        \frac{t}{\varepsilon},
        \frac{a}{\varepsilon},
        \frac{q}{\varepsilon}
    \right)
    \ge
    \frac{|q-a|^2}{2t}-K_0t
\]
for every $t>0$, and $a,q \in \R^d$. Moreover,
\[
    \widetilde u_0(a)\ge \widetilde u_0(q)-K|q-a|, \qquad a,q\in\R^d.
\]
On the other hand, taking \(a=q\) and using the constant path gives
\[
    U^\varepsilon(q,t)
    \le
    \widetilde u_0(q)+K_0t.
\]
for every $t>0$ and $q \in\R^d$. Therefore, choosing \(R>0\) sufficiently large, depending only on \(K\) and \(K_0\), we may restrict the infimum in~\eqref{eq: s7_1} to the set
\begin{equation}\label{eq: s7_2}
    |q-a|\le Rt.
\end{equation}
The same argument applies
to the effective variational formula:
\begin{equation}\label{eq: s7_3}
\widetilde u(q,t) = \inf_{a\in\mathbb R^d} \left\{\widetilde u_0(a)+\overline{m}(t,a,q) \right\}.
\end{equation}
After increasing $R$, if needed, the infimum in~\eqref{eq: s7_3} is restricted to~\eqref{eq: s7_2} with $R$ depending only on $K$ and $K_0$.

Consequently, Lemma~\ref{lem: bddestimate_mbar} and Proposition~\ref{prop: mbar identification} imply
\[
\begin{aligned}
    \left|U^\varepsilon(q,t)-\widetilde u(q,t)\right|
    \le
    \sup_{|q-a|\le Rt}
    \left|
    \varepsilon
    m\left(
        \frac{t}{\varepsilon},
        \frac{a}{\varepsilon},
        \frac{q}{\varepsilon}
    \right)
    -
    \overline m(t,a,q)
    \right|  \le C\varepsilon
\end{aligned}
\]
for every \(q\in\mathbb R^d\) and \(t>0\).

Finally, for every $x \in V$,
\[
\begin{aligned}
    |u^\varepsilon(x,t)-\widetilde u(\mathfrak m(x),t)|
    &\le
    |u^\varepsilon(x,t)-U^\varepsilon(\mathfrak m(x),t)|+
    |U^\varepsilon(\mathfrak m(x),t)-\widetilde u(\mathfrak m(x),t)| \le C\varepsilon
\end{aligned}
\]
for every $t > 0$. Hence
\[
    \|u^\varepsilon (x,t)-\widetilde u(\mathfrak m(x),t)\|_{L^\infty(V\times[0,\infty))}
    \le C\varepsilon.
\]
At $t=0$, the estimate holds trivially because $u_0(x)=\widetilde u_0(\mathfrak m(x))$.
\end{proof}
This proves the $O(\e)$ convergence rate in the convex infinite-dimensional setting.

\subsection{Optimality of the convergence rate}

We conclude by showing that the convergence rate \(O(\varepsilon)\) is optimal in general. The example below is an infinite-dimensional analogue of the one-dimensional example in \cite{Tran-Yu-optimal}.

\begin{ex}[Optimality of the \(O(\varepsilon)\) rate]
\label{ex:optimality}
Let \(d=1\), \(I=[0,1]\), and \(V=L^2(I;\mathbb R)\). Define the
\(1\)-periodic potential
\[
    W(r):=-4\cos^2(\pi r),
    \qquad r\in\mathbb R,
\]
and set
\[
    \mathcal W(x):=\int_I W(x(i))\,d\lambda_0(i),
    \qquad x\in V.
\]
Consider the Hamiltonian
\[
    H(x,p)
    :=
    \frac12\|p\|_{L^2}^2+\mathcal W(x),
    \qquad (x,p)\in V\times V.
\]
For every $0<\e<1$, let $u^\e$ be the viscosity solution to $\mathrm{(CP)}_\e$ with initial data $u_0\equiv0$. 
Then the homogenized solution \(\widetilde u\equiv0\) solves the effective equation $\mathrm{(\overline{CP})}$. Moreover, 
\begin{equation}\label{eq: s7_4}
    u^\varepsilon(0\chi_I,1)\ge \frac{\varepsilon}{6}.
\end{equation}
Consequently,
\[
    \|u^\varepsilon(x,t)-\widetilde u(\mathfrak m(x),t)\|_{L^\infty(V\times[0,\infty))}
    \ge \frac{\varepsilon}{6},
\]
and therefore the convergence rate \(O(\varepsilon)\) is optimal.
\end{ex}

\begin{proof}
Since \(W\) is smooth and \(1\)-periodic, the functional
\(\mathcal W\) is Lipschitz on \(V\). Moreover,
\[
    \mathcal W(x+z)=\mathcal W(x)
    \qquad\text{for every }z\in\Lambda,
\]
and
\[
    \mathcal W(x\circ g)=\mathcal W(x), \qquad \text{for every } g \in \mathcal{G}.
\]
Hence \(H\) is periodic and rearrangement invariant. It is also locally Lipschitz, coercive, and convex in the momentum variable. Thus \rm{(H1)--(H5)} hold. The initial datum \(u_0\equiv0\) clearly satisfies \rm{(I1)--(I2)}.

The corresponding Lagrangian is
\begin{equation*}
    L(x,v)
    =
    \frac12\|v\|_{L^2}^2-\mathcal W(x).
\end{equation*}
Since \(W\le0\), we have \(L\ge0\).

We first identify the homogenized solution. Let \(r_*=1/2\). Since $W(r_*)=0$, for all sufficiently large \(t\), we may construct a curve that moves from \(0\chi_I\) to \(r_*\chi_I\) in unit time, remains at \(r_*\chi_I\) until time \(t-1\), and then returns to \(0\chi_I\) in unit time. The action of this curve is bounded independently of \(t\), because $L(r_*\chi_I,0)=0$. Therefore,
\[
    0\le m(t,0,0) \le C
\]
for all sufficiently large \(t\), and hence
\begin{equation*}
    \overline L(0)
    =
    \overline m(1,0,0)
    =
    \lim_{t\to\infty}\frac1t m(t,0,0)
    =
    0.
\end{equation*}

Since $\overline{H}$ is even and convex (see Remark~\ref{rem2}), it attains its minimum at \(p=0\). By Legendre duality,
\[
    0 = -\overline L(0)
    =
    -\sup_{p\in\mathbb R}\{-\overline H(p)\}
    =
    \min_{p\in\mathbb R}\overline H(p) = \overline{H}(0).
\]
It follows that
\[
    \overline H(0)=0.
\]
Therefore \(\widetilde u\equiv0\) is the solution of $\mathrm{(\overline{CP})}$ with initial
data \(u_0\equiv0\).

It remains to prove \eqref{eq: s7_4}. By time reversal, using $L(x,v)=L(x,-v)$ for $x,v \in V$, the optimal-control representation becomes
\begin{equation}\label{eq: s7_5}
    u^\varepsilon(0\chi_I,1)
    =
    \inf_{\substack{
        \gamma\in AC([0,\varepsilon^{-1}];V)\\
        \gamma(0)=0\chi_I
    }}
    \varepsilon
    \int_0^{\varepsilon^{-1}}
    \left(
        \frac12\|\dot\gamma(s)\|_{L^2}^2
        -
        \mathcal W(\gamma(s))
    \right)\,ds .
\end{equation}

Fix an admissible curve \(\gamma\) and choose the standard representative described in Remark~\ref{rem3}. Then, for almost every \(i\in I\), the scalar path
\[
    \gamma_i(s):=\gamma(s)(i)
\]
is absolutely continuous and satisfies \(\gamma_i(0)=0\). By Tonelli's theorem, the action in \eqref{eq: s7_5} is
\[
    \varepsilon
    \int_I
    \int_0^{\varepsilon^{-1}}
    \left(
        \frac12|\dot\gamma_i(s)|^2
        -
        W(\gamma_i(s))
    \right)\,ds\,d\lambda_0(i).
\]

We claim that, for almost every \(i\in I\),
\begin{equation}\label{eq: s7_6}
    L^\gam_i := \int_0^{\varepsilon^{-1}}
    \left(
        \frac12|\dot\gamma_i(s)|^2
        -
        W(\gamma_i(s))
    \right)\,ds
    \ge \frac16.
\end{equation}
Indeed, recall that
\[
    W(r)\le-1
    \qquad\text{for } |r|\le\frac13.
\]

There are two cases. If
\[
    \gamma_i\left([0,\tfrac13]\right)
    \subset
    \left[-\tfrac13,\tfrac13\right],
\]
then
\begin{equation*}
     L^\gam_i
    \ge
    \int_0^{1/3}-W(\gamma_i(s))\,ds \ge \frac13.
\end{equation*}
Otherwise, by continuity there exists \(\tau_i\in(0,\frac13]\) such that $|\gamma_i(\tau_i)|=1/3$. Since \(-W\ge0\), the Cauchy--Schwarz inequality gives
\begin{equation*}
     L^\gam_i \ge
    \frac12\int_0^{\tau_i}|\dot\gamma_i(s)|^2\,ds \ge
    \frac{1}{2\tau_i}
    \left|
        \int_0^{\tau_i}\dot\gamma_i(s)\,ds
    \right|^2 \ge
    \frac16.
\end{equation*}
Thus \eqref{eq: s7_6} holds in either case.

Since this holds for every admissible curve \(\gamma\), integrating \eqref{eq: s7_6} over \(I\) and taking the infimum in \eqref{eq: s7_5} yields
\[
    u^\varepsilon(0\chi_I,1)\ge\frac{\varepsilon}{6}.
\]
Because \(u\equiv0\), the conclusion follows.
\end{proof}
\begin{rem}
\label{rem2}
For the Hamiltonian in Example~\ref{ex:optimality}, the effective Hamiltonian $\overline H$ is even. 
Since $L(x,v)=L(x,-v)$ for $x,v \in V$, time reversal gives
\[
    m(t,a,q)=m(t,q,a).
\]
Therefore, by Lemma~\ref{lem: mbarproperty},
\[
\overline L(q)
=\overline m(1,0,q)
=\overline m(1,q,0)
=\overline m(1,0,-q)
=\overline L(-q).
\]
Thus $\overline L$ is even, and hence so is its Legendre transform $\overline H$.

Moreover, the convexity of $\overline{H}$ follows from the metric identification established in Proposition~\ref{prop: mbar identification}.
\end{rem}

\begin{rem}\label{rem3}
Let
\[
    \gamma\in AC([0,T];V).
\]
By the fundamental theorem of calculus for absolutely continuous functions taking values in a Hilbert space, there exists
\[
    \dot\gamma\in L^1([0,T];V)
\]
such that
\[
    \gamma(t)=\gamma(0)+\int_0^t\dot\gamma(s)\,ds
    \qquad\text{in }V
\]
for every $t\in[0,T]$.

Choose a jointly measurable representative
\[
    \dot\Gamma:[0,T]\times I\to\mathbb R^d
\]
of $\dot\gamma$, and choose a measurable representative
$\Gamma_0:I\to\mathbb R^d$ of $\gamma(0)$. Define
\[
    \Gamma(t,i):=
    \Gamma_0(i)+\int_0^t\dot\Gamma(s,i)\,ds.
\]
Since $\lambda_0(I)=1$,
\[
\begin{aligned}
    \int_0^T\int_I|\dot\Gamma(s,i)|\,d\lambda_0(i)\,ds \le
    \int_0^T
    \|\dot\Gamma(s,\cdot)\|_{L^2(I)}\,ds <\infty.
\end{aligned}
\]
Fubini's theorem therefore implies that $s\mapsto\dot\Gamma(s,i)$ is integrable for almost every $i\in I$. Consequently, for almost every $i$, the path
\[
    \gamma_i(t):=\Gamma(t,i)
\]
is absolutely continuous and satisfies
\[
    \dot\gamma_i(t)=\dot\Gamma(t,i).
\]
Moreover, $\Gamma(t,\cdot)$ represents $\gamma(t)$ for every $t\in[0,T]$. If $\gamma(0)=0$, we may choose $\Gamma_0=0$, and hence $\gamma_i(0)=0$ for almost every $i$.

For the Lagrangian in Example~1, the function
\[
    (s,i)\longmapsto
    \frac12|\dot\Gamma(s,i)|^2-W(\Gamma(s,i))
\]
is jointly measurable and nonnegative because $W\le0$. Tonelli's theorem therefore gives
\[
\begin{aligned}
    &\int_0^T
    \left(
        \frac12\|\dot\gamma(s)\|_{L^2}^2
        -\mathcal W(\gamma(s))
    \right)\,ds =
    \int_I\int_0^T
    \left(
        \frac12|\dot\gamma_i(s)|^2
        -W(\gamma_i(s))
    \right)\,ds\,d\lambda_0(i)
\end{aligned}
\]
where the two sides are allowed a priori to take the value \(+\infty\). For a curve with infinite action, the lower bound required in Example~1 is automatic. Hence, it is enough to consider finite-action curves.
\end{rem}

\bibliographystyle{naturemag}
\bibliography{references}

\end{document}